%% file: witt_invariants_quaternionic_forms.tex
\documentclass[10pt,a4paper]{article}

\include{canevas_anglais}

\author{Nicolas Garrel}
\title{Witt invariants of quaternionic forms}
\date{}

\begin{document}

\maketitle

\section*{Introduction}

In the seminal \cite{GMS}, Serre starts the study of cohomological
invariants, in particuliar of algebraic groups. Let $k$ be a base field,
which in this article we will assume to be of characteristic not $2$.
Then if $\mathbf{Field}_{/k}$ is the category of field extensions of $k$
(it is enough to consider finitely generated extensions, in case one wants
to work in an essentially small category), and $F: \mathbf{Field}_{/k}\to \mathbf{Set}$
and $A: \mathbf{Field}_{/k}\to \mathbf{Ab}$ are functors, then an 
invariant of $F$ with values in $A$ is simply a natural transformation
from $F$ to $A$ (seen as a functor to $\mathbf{Set}$). We write 
$\Inv(F,A)$ for the set of such invariants; this is clearly an abelian 
group for the pointwise addition coming from the abelian group structure
of $A(K)$ for any extension $K/k$.

If $A(K) = H^i(K,M)$ for some Galois module $M$ defined over $k$, we speak of 
cohomological invariants with values in $M$. If $F(K)=H^1(K,G)$ for some algebraic 
group $G$ defined over $k$, we speak of "invariants of $G$" (a more explicit terminology 
would be "invariants of $G$-torsors"). If $G$ is a classical group,
one recovers some familiar functors; in particular, if $G = O_n$ is the
classical orthogonal group, the corresponding $F$ is isomorphic to the 
functor $\mathbf{Quad}_n$ of non-singular $n$-dimensional quadratic forms.
Taking $M = \Z/2\Z$ with trivial Galois action, the classical invariants
of quadratic forms $q$ such as the determinant $\det(q)\in K^\times/(K^\times)^2
\simeq H^1(K,\Z/2\Z)$ and the Hasse invariant in $\Br(K)[2]\simeq H^2(K,\Z/2\Z)$
give examples of cohomological invariants of $O_n$. In \cite{GMS} Serre 
gives a complete description of mod $2$ cohomological invariants of $O_n$,
proving that they form a free module over $H^*(k,\Z/2\Z)$ with basis 
given by the Stiefel-Whitney invariants $w_i$ (the determinant and Hasse 
invariant are then recovered as, respectively, $w_1$ and $w_2$).

Serre also introduces in \cite{GMS} the notion of Witt invariants
(which amounts to taking $A(K)=W(K)$ the Witt group of $K$),
noticing that the formal properties of the Witt group regarding 
residues and specialization with respect to discrete valuations are similar 
enough to cohomology groups that a similar theory can be developped.
For Witt invariants of $O_n$, he proves that they form a free $W(k)$-module
generated by the $\lambda$-operations $\lambda^i$ (for $0\ppq i\ppq n$).

Although there is clearly much more literature devoted to cohomological invariants,
in particular due to their use for rationality problems and computations
of essential dimensions of algebraic groups, the solution of the Milnor
Conjecture by Voevodsky gives a direct connexion between cohomological 
and Witt invariants which can motivate a closer interest in those.
Explicitly, if we can define a Witt invariant $\alpha\in \Inv(F,W)$ with values in 
$I^n\subset W$ (where $I^n$ is the $n$th power of the fundamental ideal of the Witt
group), then composing with the canonical isomorphism 
$I^n(K)/I^{n+1}(K)\isom H^2(K,\Z/2\Z)$ defines a cohomological 
invariant in $\Inv^n(F,\Z/2\Z)$. It is not difficult to realize 
that one recovers all cohomological invariants of $O_n$ 
in this way (see \cite[Section 9]{G3} for instance)

Now we can ask the question of invariants of $G=O(A,\sigma)$ where 
$(A,\sigma)$ is a central simple algebra with orthogonal involution.
This algebraic group is a form of $O_n$ (which corresponds to the case 
where $A=M_n(k)$ and $\sigma$ is the adjoint involution of the form 
$\fdiag{1,\dots,1}$), so we might expect a similar description of invariants.
The pointed set $H^1(K,G)$ is this time is a canonical bijection with the 
set of isometry classes of nondegenerate $1$-dimensional hermitian forms over 
$(A,\sigma)$ (the base point being $\fdiag{1}_\sigma$). We address in this 
article the question of Witt invariants of $O(A,\sigma)$ when $A$ has 
index $\ppq 2$ and degree $2r$ (which is a class stable by scalar extension). 
If $Q$ is the quaternion algebra Brauer-equivalent to
$A$, we can choose a Morita equivalence between $(A,\sigma)$ and $(Q,\gamma)$
where $\gamma$ is the canonical symplectic involution of $Q$,
and this induces an isomorphism between the functor $H^1(-,G)$
and $H_Q^{(r)}$, where for any extension $K/k$, $H_Q^{(r)}(K)$
is the set of isometry classes of nondegenerate anti-hermitian forms 
of reduced dimension $2r$ over $(Q_K,\gamma_K)$. Our main result is then 
Theorem \ref{thm_inv}, which states that the Witt invariants of $H_Q^{(r)}(K)$
(and therefore of $G$) are again generated by $\lambda$-operations (this 
time in the sense of \cite{G2}), but the coefficients have to be taken
not only in $W(k)$ but in the mixed Witt ring $\tld{W}^{-1}(Q,\gamma)$
introduced in \cite{G1}. Furthermore, such a decomposition is not
exactly unique, the norm form $n_Q$ being the obstruction.

\section*{Notations and conventions}

If $Q$ is a quaternion algebra, $\gamma$ is its canonical symplectic involution,
$Q_0$ is the space of pure quaternions, and $Q_0^\times$ the set of 
invertible pure quaternions.

\section{Mixed Witt rings}

\subsection{General case}

Let $(A,\sigma)$ be an Azumaya algebra with involution of the first kind over $k$,
and $\eps\in \mu_2(k)$. Then $SW^\eps(A,\sigma)$ is defined as the monoid of $\eps$-hermitian
forms over $(A,\sigma)$, with orthogonal sums, $GW^\eps(A,\sigma)$ is the Grothendieck 
group of $SW^\eps(A,\sigma)$, and $W^\eps(A,\sigma)$ is the quotient of $GW^\eps(A,\sigma)$
by the subgroup of hermitian forms.

For any $a\in A^\times$ such that $\sigma(a)=\eps a$, the elementary diagonal
form $\fdiag{a}_\sigma$ is 
\begin{equation}
    \foncdef{\fdiag{a}_\sigma}{A\times A}{A}{(x,y)}{\sigma(x)ay.}
\end{equation}
The diagonal form $\fdiag{a_1,\dots,a_r}_\sigma$ is the orthogonal sum 
\begin{equation}
    \fdiag{a_1,\dots,a_r}_\sigma = \fdiag{a_1}_\sigma \perp \dots \perp 
    \fdiag{a_r}_\sigma.
\end{equation}

When $(A,\sigma)=(k,\Id)$ and $\eps=1$ we just write $SW(k)$, $GW(k)$ and
$W(k)$. Note that $W^\eps(A,\sigma)$ is naturally a $W(k)$-module. We define
a $\Zd$-graded $GW(k)$-module 
\begin{equation}
    \tld{GW}^{\eps}(A,\sigma) = GW(k) \oplus GW^{\eps}(A,\sigma)
\end{equation}
and a $\Zd$-graded $W(k)$-module 
\begin{equation}
    \tld{W}^{\eps}(A,\sigma) = W(k) \oplus W^{\eps}(A,\sigma).
\end{equation}

If $(A,\sigma)$ and $(B,\tau)$ are Azumaya algebra with involution of 
the first kind over $k$, and $\eps_0\in \mu_2(k)$, we write 
\begin{equation}\label{eq_morita_eq}
    (B,\tau) \stackrel{(V,h)}{\rightsquigarrow} (A,\sigma) 
\end{equation}
if $V$ is a $B$-$A$-bimodule, balanced over $k$, such that 
$B\simeq \End_A(V)$ for this action, and $h$ is an $\eps_0$-hermitian 
form over $(A,\sigma)$, and $\tau$ is the adjoint involution of $h$.
We then say that (\ref{eq_morita_eq}) is an $\eps_0$-hermitian 
Morita equivalence, or simply a Morita equivalence. Such an equivalence
induces graded isomorphisms $h_*$ which fit in this commutative 
diagram
\begin{equation}\label{eq_diagram_morita}
    \begin{tikzcd}
        \tld{GW}^{\eps}(B,\tau) \rar{h_*} \dar & \tld{GW}^{\eps\eps_0}(A,\sigma) \dar \\ 
        \tld{W}^{\eps}(B,\tau) \rar{h_*}  & \tld{W}^{\eps\eps_0}(A,\sigma) 
    \end{tikzcd}
\end{equation}
such that $h_*$ is the identity on the even components $GW(k)$ and $W(k)$.

In \cite{G1}, a graded commutative ring structure is defined on $\tld{GW}^{\eps}(A,\sigma)$
and $\tld{W}^{\eps}(A,\sigma)$ such that (\ref{eq_diagram_morita}) is 
a commutative diagram of rings. The product is characterized by 
\begin{equation}\label{eq_prod_elem}
    \fdiag{a}_\sigma\cdot \fdiag{b}_\sigma \simeq T_{\sigma,a,b}\in SW(k)
\end{equation}
where $T_{\sigma,a,b}$ is the twisted involution trace form defined as 
\begin{equation}\label{eq_twisted_trace}
    \foncdef{T_{\sigma,a,b}}{A\times A}{k}{(x,y)}{\Trd_A(\sigma(x)ay\sigma(b)).}
\end{equation}

Furthermore, if $K/k$ is a field extension, then the 
scalar extension maps induce a commutative diagram
\begin{equation}\label{eq_diagram_scalar_ext}
    \begin{tikzcd}
        \tld{GW}^{\eps}(B,\tau) \rar{h_*} \dar & \tld{GW}^{\eps\eps_0}(A,\sigma) \dar \\ 
        \tld{GW}^{\eps}(B_K,\tau_K) \rar{(h_K)_*}  & \tld{GW}^{\eps\eps_0}(A_K,\sigma_K) 
    \end{tikzcd}
\end{equation}
and similarly for mixed Witt rings.

If $(A,\sigma)=(k,\Id)$ and $\eps=1$, the canonical $\Zd$-graded $W(k)$-module
isomorphism
\begin{equation}\label{eq_witt_split}
    \tld{W}^1(k,\Id) = W(k) \oplus W(k) \simeq W(k)[\Zd]
\end{equation}
is a graded $W(k)$-algebra isomorphism. Let us write 
\begin{equation}\label{eq_delta_witt}
    \delta: \tld{W}^1(k,\Id)\to W(k)
\end{equation}
for the map $W(k)\oplus W(k)\to W(k)$ given by the sum 
of components. It is a $W(k)$-algebra morphism.

\subsection{Quaternion algebras}

We consider the case of a quaternion algebra $Q$ with its canonical involution $\gamma$,
and $\eps=-1$. Then for any invertible pure quaternions $z_1,z_2\in Q_0^\times$,
a direct computation of the form $T_{\gamma,z_1,z_2}$ in (\ref{eq_twisted_trace})
(see \cite[Prop 4.12]{G1}) shows that
we have in $\tld{W}^{-1}(Q,\gamma)$:
\begin{equation}\label{eq_prod_elem_quat}
    \fdiag{z_1}_\gamma \cdot \fdiag{z_2}_\gamma = \fdiag{-\Trd_Q(z_1z_2)}\phi_{z_1,z_2}\in W(k)
\end{equation}
where $\phi_{z_1,z_2}$ is the unique $2$-fold Pfister form whose Witt class 
is $\pfis{z_1^2,z_2^2}-n_Q\in W(k)$. If $\Trd_Q(z_1z_2)=0$ (which means that 
$z_1$ and $z_2$ anti-commute),
(\ref{eq_prod_elem_quat}) should be understood as saying that 
$\fdiag{z_1}_\gamma \cdot \fdiag{z_2}_\gamma=0\in W(k)$.

When $Q$ is not split, (\ref{eq_prod_elem_quat}) entirely characterizes the $W(k)$-algebra
structure of $\tld{W}^{-1}(Q,\gamma)$ since $W^{-1}(Q,\gamma)$ is additively 
generated by the $\fdiag{z}_\gamma$. When $Q$ is split, we may choose
$z_0\in Q\setminus \{0\}$ such that $z_0^2=0$. Then the left ideal $Qz_0$ is a 
$2$-dimensional $k$-vector
space, and if we define the anti-symmetric bilinear form 
\begin{equation}\label{eq_b_z0}
    \foncdef{b_{z_0}}{Qz_0\times Qz_0}{k}{(zz_0,z'z_0)}{-z_0\gamma(z)z',}
\end{equation}
or equivalently
\begin{equation}
    b(x,y)z_0 = \gamma(x)y
\end{equation}
for all $x,y\in Qz_0$, we get an anti-hermitian Morita equivalence
\begin{equation}
    (Q,\gamma) \stackrel{(Qz_0,b_{z_0})}{\rightsquigarrow} (k,\Id).
\end{equation}

This induces a $W(k)$-algebra morphism 
\begin{equation}\label{eq_def_phi}
    \Phi_{z_0}: \tld{W}^{-1}(Q,\gamma) \stackrel{(b_{z_0})_*}{\isom} \tld{W}^1(k,\Id)
    \isom W(k)[\Zd] \xrightarrow{\delta} W(k)
\end{equation}
using (\ref{eq_diagram_morita}), (\ref{eq_witt_split}) and (\ref{eq_delta_witt}). Note 
that the restriction of $\Phi_{z_0}$ to $W(k)$ is the identity.

\begin{lem}\label{lem_morita_split_quat}
    Let $z_0\in Q_0 \setminus \{0\}$ be such that $z_0^2=0$. Then for 
    any $z\in Q_0^\times$, $(b_{z_0})_*(\fdiag{z}_\gamma)$ is 
    isometric to the symmetric bilinear form 
    \begin{equation}
        \foncdef{b_{z_0,z}}{Qz_0\times Qz_0}{k}{(z_1z_0,z_2z_0)}{-\Trd_Q(z_0\gamma(z_1)zz_2).}
    \end{equation}
    If $z$ and $z'$ anti-commute, this form is a hyperbolic plan; otherwise,
    it is isometric to $\fdiag{-\Trd_Q(zz_0)}\pfis{z^2}$.
\end{lem}

\begin{proof}
    From the general theory of hermitian Morita equivalences, $(b_{z_0})_*$
    sends the anti-hermitian space $(Q,\fdiag{z}_\gamma)$ to 
    \[ \anonfoncdef{Q\otimes_Q Qz_0\times Q\otimes_Q Qz_0}{K}
    {(z_1\otimes z_0,z_2\otimes z_0)}{b_{z_0}(z_0,\fdiag{z}_\gamma(z_1,z_2)z_0).} \]
    Now 
    \begin{align*}
        b_{z_0}(z_0,\fdiag{z}_\gamma(z_1,z_2)z_0)z_0 &= -z_0(\gamma(z_1)zz_2)z_0
        &= -\Trd_Q(z_0\gamma(z_1)zz_2)z_0
    \end{align*}
    where the last equality is because for any $x\in Q$ we have 
    \[ \Trd_Q(z_0x)z_0 = (z_0x - \gamma(x)z_0)z_0 = z_0xz_0.  \]
    If $z$ and $z_0$ anti-commute, $zz_0\neq 0$ and $b_{z_0,z}(zz_0,zz_0)=0$,
    so $b_{z_0,z}$ is isotropic, and therefore a hyperbolic plan.

    If $z$ and $z_0$ do not anti-commute, we have $\Trd_Q(zz_0)\neq 0$,
    and $(z_0, zz_0)$ is an orthogonal $k$-basis of $Qz_0$ for $b_{z_0,z}$,
    which gives the diagonalization $b_{z_0,z}\simeq \fdiag{-\Trd_Q(zz_0),\Trd_Q(zz_0)z^2}$.
\end{proof}

\subsection{$\lambda$-operations}

In \cite{G2}, for any Azumaya algebra with involution of the first kind
$(A,\sigma)$ over $k$, and any $\eps\in \mu_2(k)$, a structure of  
pre-$\lambda$-ring (see \cite{Yau} for a reference about 
pre-$\lambda$-rings) is defined on $\tld{GW}^\eps(A,\sigma)$, whose 
restriction to $GW(K)$ is the usual $\lambda$-ring structure (studied for 
instance in \cite{McGar}).

It is compatible with Morita equivalences, meaning that the top
row of (\ref{eq_diagram_morita}) is an isomorphism of pre-$\lambda$-rings,
and it is compatible with scalar extensions, meaning the (\ref{eq_diagram_scalar_ext})
is a commutative diagram of pre-$\lambda$-rings.

Note that the pre-$\lambda$-ring structure is compatible with the $\Zd$-grading,
meaning that $\lambda^d(GW^\eps(A,\sigma))$ is included in $GW(k)$ when $d$ 
is even, and in $GW^\eps(A,\sigma)$ when $d$ is odd. Also note that by definition
of a pre-$\lambda$-ring, $\lambda^0$ is the constant function to $\fdiag{1}$,
and $\lambda^1$ is the identity.

It follows from \cite[Prop 5.2]{G2} that if $\sigma$ is symplectic, $a\in A^\times$ 
is $\eps$-symmetric and $n=\deg(A)$, then 
\begin{equation}\label{eq_det_nrd}
    \lambda^n(\fdiag{a}_\sigma) = \fdiag{\Nrd_A(a)}.
\end{equation}
The square class defined by this $1$-dimensional form is precisely
the determinant of $\fdiag{a}_\sigma$ (as defined in \cite{BOI}).


\section{Generic splitting and residues}

A crucial method for us is the scalar extension to a generic splitting field of 
our quaternion algebra. The behaviour of anti-hermitian forms under such generic 
splitting has been the object of a fair amount of research, but we will mainly 
refer to \cite{QT}, which presents a good overview of the situation.

\subsection{The generic elementary form}

Let $Q$ be a quaternion algebra. We choose a quaternionic 
basis $(i,j,ij)$, with $i^2=a$ and $j^2=b$, such that $(ij)^2$ is not
a square in $k$. This is of course automatic when $Q$ is not split,
and even when $Q$ is split it is always possible unless $k$ is quadratically
closed (we exclude this case from the present discussion).

We define the generic pure quaternion of $Q$ as
\begin{equation}\label{eq_def_omega_gen}
    \tld{\omega} = xi + yj + zij \in Q_{k(x,y,z)}.
\end{equation}

To make use of the fact that $\tld{\omega}$ is generic, we 
use the setting of versal torsors as in \cite[Section 5]{GMS}. Let 
$h_0\in H_Q^{(1)}(k)$ and $G=O(h_0)$ be its orthogonal group.
There is a canonical isomorphism $h\mapsto \Iso(h,h_0)$ between
the functors $H_Q^{(1)}$ and $K\mapsto H^1(K,G)$ which allows us 
to view elementary forms as $G$-torsors. Then 
$\fdiag{\tld{\omega}}_{\gamma_{k(x,y,z)}}\in H_Q^{(1)}(k(x,y,z))$
is a torsor over the function field of $\mathbb{A}^3_k$, and it 
is the generic point of a torsor over $\mathbb{A}^3_k\setminus V(\tld{\omega}^2)$
(because the specialization of $\tld{\omega}$ at some point in $\mathbb{A}^3_k$
is non-invertible if and only if this point is in $V(\tld{\omega}^2)$).
Note that $\tld{\omega}^2\in k[x,y,z]$ is nothing but the pure norm form 
of $Q$ in the coordinate system given by the basis $(i,j,ij)$ of $Q_0$.

\begin{lem}\label{lem_versal}
    The $G$-torsor over $k(x,y,z)$ corresponding to  
    $\fdiag{\tld{\omega}}_{\gamma_{k(x,y,z)}}$ is a versal $G$-torsor.
\end{lem}

\begin{proof}
    Let $K/k$ be a field extension with $K$ infinite, and let 
    $h\in H_Q^{(1)}(K)$. Clearly the points $(s,t,r)\in \mathbb{A}^3(K)$
    such that $h\simeq \fdiag{si+tj+rij}_{\gamma_K}$ are dense,
    so for any open $U\subset \mathbb{A}^3_k$ there is a point 
    in $U(K)$ such that $h$ is the corresponding specialization 
    of $\fdiag{\tld{\omega}}_{\gamma_{k(x,y,z)}}$ seen as a torsor on 
    $\mathbb{A}^3\setminus V(\tld{\omega}^2)$.
\end{proof}

We also define 
\begin{equation}\label{eq_def_omega}
    \omega = xi+yj+ij \in Q_{k(x,y)}
\end{equation}
and 
\begin{equation}\label{eq_def_delta}
    \Delta = -\omega^2 = -ax^2-by^2+ab =  \in k[x,y].
\end{equation}

\subsection{The Severi-Brauer variety}

Let $SB(Q)$ be the Severi-Brauer variety of $Q$. By definition, 
if $K/k$ is an extension, $SB(Q)(K)$ is the set of left ideals 
of reduced dimension $1$ (equivalently, of $K$-dimension $2$) of $Q_K$.
If $I\subset Q_K$ is such an ideal, then $I=Qz_0$ for some 
non-zero pure quaternion $z_0\in Q_K$ with $z_0^2=0$, and $z_0$ is 
unique up to a constant. Thus all such $z_0$ lie on a line $L(I)$
which is recovered intrinsically as $L(I)=I\cap \gamma_K(I)$.

If $X_Q$ is the projective conic defined by the pure norm form of 
$Q$, $X_Q(K)$ is the set of lines in $(Q_K)_0$ consisting of pure 
quaternion whose square is $0$, and there is a canonical 
isomorphism $SB(Q)\simeq X_Q$ sending $I\in SB(Q)(K)$ to 
$L(I)\in X_Q(K)$.

Let $F_\infty$ be the quadratic extension $k(ij)\subset Q$ 
of $k$. Let $V=ki\oplus kj\subset Q_0$; we define $\mu: V\otimes_k F_\infty\to V$
as the multiplication map inside $Q$, and $L_\infty = \ker(\mu)\subset Q_0\otimes_k F_\infty$.
Then one may chech that $L\infty$ is a point in $X_Q(F_\infty)$, which 
defines a closed point $\infty\in X_Q$ of degree $2$ and residue field 
$F_\infty$.

Then if 
\begin{equation}\label{eq_y}
    Y = V(\Delta) \subset \mathbb{A}^2_k
\end{equation}
is the affine conic defined by $\Delta$, there is a natural identification 
$Y\simeq X_Q\setminus \{\infty\}$, and therefore
\begin{equation}\label{eq_f}
    F = \operatorname{Frac}(k[x,y]/(\Delta)) = k(Y)
\end{equation}
is a function field of $X_Q$, and thus a generic splitting 
field of $Q$.

The image of $\omega\in Q\otimes_k k[x,y]$ in $Q_F$ is written 
$\bar{\omega}$. By definition, of $F$ $\bar{\omega}^2=0$ (which is 
a witness to the fact that $F$ is a splitting field of $Q$).

\subsection{Valuations and residues}

Let $K/k$ be a field extension
and $v: K^\times\to \Z$ a valuation on $K$ which is trivial on $k$, with valuation ring 
$\mathcal{O}_v$, uniformizing element $\pi$ and residue field $\kappa_v$.
Recall that there are 
residue maps $\partial^1_v: W(K)\to W(\kappa_v)$ (independent of $\pi$) and 
$\partial^2_{v,\pi}: W(K)\to W(\kappa_v)$
(which depends on the choice of $\pi$, but its kernel doesn't). They actually form 
a $W(k)$-algebra morphism $\partial_{v,\pi}: W(K) \to W(\kappa_v)[\Zd]$
(where the even component is $\partial_v^1$ and the odd component is 
$\partial_{v,\pi}^2$).
In practice, if $q\in W(K)$, we can write 
$q = \fdiag{a_1,\dots,a_n}+\fdiag{\pi}\fdiag{b_1,\dots,b_m}$
with $a_i,b_i\in \mathcal{O}_v$ for all $i$ and $j$, and then 
$\partial^1(q)=\fdiag{\bar{a_1},\dots,\bar{a_n}}$
and $\partial_{2,\pi}(q) = \fdiag{\bar{b_1},\dots,\bar{b_m}}$.
\\

Every closed point $p\in Y^{(1)}$ defines a discrete rank $1$ valuation $v_p$
on $F$, with residue field $F_p$.  We write
\begin{equation}
    W_0(F) = \bigcap_{p\in Y^{(1)}} \ker(\partial^2_{v_p,\pi_p}: W(F)\to W(F_p))
\end{equation}
which does not depend on the choice of uniformizers $\pi_p$ for each $p\in Y^{(1)}$.
It is a sub-$W(k)$-algebra of $W(F)$.

There is also the valuation "at infinity" $v_\infty$
corresponding to the point $\infty\in X_Q^{(1)}$, with residue field $F_\infty$. 
It is characterized by 
the fact that if $\bar{u}\in k[Y] = k[x,y]/(\Delta)$ is the class of $u=k[x,y]$,
then $v_\infty(\bar{u})=-\deg(u)$. We will shorten $\partial^1_{v_\infty}$ 
and $\partial^2_{v_\infty,\pi_\infty}$ as $\partial^1_\infty$ and $\partial^2_\infty$,
where $\pi_\infty$ is any choice of uniformizer (which will not matter to us).

\subsection{Generic splitting of hermitian forms}

Since $\bar{\omega}^2=0$ in $Q_F$, we get a $W(k)$-algebra morphism 
\begin{equation}
    \Phi_{\bar{\omega}}: \tld{W}^{-1}(Q_F,\gamma_F)\to W(F)
\end{equation}
(see (\ref{eq_def_phi})), and its composition with the scalar extension map 
yields a $W(k)$-algebra morphism
\begin{equation}
    \Psi_{\bar{\omega}}: \tld{W}^{-1}(Q,\gamma)\to \tld{W}^{-1}(Q_F,\gamma_F) 
    \xrightarrow{\Phi_{\bar{\omega}}} W(F).
\end{equation}

By definition, the restriction of $\Psi_{\bar{\omega}}$ to $W(k)$ is the 
scalar extension map $W(k)\to W(F)$, and its restriction to $W^{-1}(Q,\gamma)$
is the composition of the scalar extension map to $F$ with the isomorphism
$(b_{\bar{\omega}})_*$ (see (\ref{eq_b_z0})).

The exact sequences in \cite[Thm 5.1, Thm. 5.2]{QT} have the following exact 
sequences as direct consequences:
\begin{equation}\label{eq_sequence_even}
    0 \to n_QW(k) \to W(k) \xrightarrow{\Psi_{\bar{\omega}}} W_0(F) 
    \xrightarrow{\partial^2_\infty} W(F_\infty) 
\end{equation}
\begin{equation}\label{eq_sequence_odd}
    0 \to W^{-1}(Q,\gamma) \xrightarrow{\Psi_{\bar{\omega}}} W_0(F)
    \xrightarrow{\partial^1_\infty} W(F_\infty) 
    \xrightarrow{s_*} W(k)
\end{equation}
where $s: F_\infty\to k$ is any $k$-linear form which is $0$ on $k$.
We collect some immediate observations on these sequences:

\begin{prop}\label{prop_ker_psi}
    We have $W(k)\cap \Ker(\Psi_{\bar{\omega}}) = n_QW(k)$
    and $W^{-1}(Q,\gamma)\cap \Ker(\Psi_{\bar{\omega}}) = 0$.
    The scalar extension map $\tld{W}^{-1}(Q,\gamma)\to 
    \tld{W}^{-1}(Q_F,\gamma_F)$ has kernel $n_QW(k)$.
\end{prop}

\begin{coro}\label{cor_nq_trivial}
    We have $n_Q W^{-1}(Q,\gamma) = 0$, and $n_Q \tld{W}^{-1}(Q,\gamma) = n_Q W(k)$.
\end{coro}

\begin{proof}
    Since $\Psi_{\bar{\omega}}(n_Q)=n_{Q_F}=0$, $\Psi_{\bar{\omega}}(n_Q W^{-1}(Q,\gamma))=0$, 
    so $n_Q W^{-1}(Q,\gamma)$ since $\Psi_{\bar{\omega}}$ is injective on 
    $W^{-1}(Q,\gamma)$.
\end{proof}

\begin{rem}\label{rem_nq_module}
    In particular, this means that the $W(k)$-module structure of $W^{-1}(Q,\gamma)$
    factors through a $W(k)/(n_Q)$-module structure.
\end{rem}

A more precise description of the image and kernel of $\Psi_{\bar{\omega}}$ is:

\begin{prop}\label{prop_exact_gen_splitting}
    The scalar extension map $\tld{W}^{-1}(Q,\gamma)\to 
    \tld{W}^{-1}(Q_F,\gamma_F)$ has kernel $n_QW(k)$, and 
    there is an exact sequence 
    \[ 0 \to (\fdiag{2}\pfis{(ij)^2}-\fdiag{ij}_\gamma)\tld{W}^{-1}(Q,\gamma) 
     \to \tld{W}^{-1}(Q,\gamma) 
    \xrightarrow{\Psi_{\bar{\omega}}} W_0(F) \to 0. \]
\end{prop}

\begin{proof}
    From the exact sequences (\ref{eq_sequence_even}) and  
    (\ref{eq_sequence_odd}), we see that 
    the kernel of $W(k)\to W(F)$ is $n_QW(k)$, that $W^{-1}(Q,\gamma)
    \to W^{-1}(Q_F,\gamma_F)$ is injective, and that $\Psi_{\bar{\omega}}$
    has image in $W_0(F)\subset W(F)$.

    Let us show that $W_0(F)$ is included in the image of $\Psi_{\bar{\omega}}$.
    Let $q\in W_0(F)$, and write $q_\infty=\partial^1_\infty(q)\in W(F_\infty)$. 
    From (\ref{eq_sequence_odd}), we get that $s_*(q_\infty)=0\in W(k)$.
    But since $F_\infty$ is quadratic extension of $k$, the scalar extension map
    and $s_*$ fit in an exact sequence 
    \[  W(k) \xrightarrow{\rho} W(F_\infty) \xrightarrow{s_*} W(k) \]
    by \cite[Thm 34.4]{EKM}. Thus $q_\infty=\rho(q_0)$ for some 
    $q_0\in W(k)$. Now let $q_1=q-\Psi_{\bar{\omega}}(q_0)\in W_0(F)$. Then
    \[ \partial^1_\infty(q_1) = \partial^1_\infty(q) - \rho(q_0) = 0  \]
    where we used that the composition 
    \[ W(k)\xrightarrow{\Psi_{\bar{\omega}}} W(F) \xrightarrow{\partial^1_\infty} W(F_\infty)  \]
    is nothing but $\rho$ by definition of $\partial^1_\infty$.  
    Using (\ref{eq_sequence_odd}), we see that $q_1=\Psi_\omega(h_1)$ for some 
    $h_1\in W^{-1}(Q,\gamma)$. In the end $q = \Psi_{\bar{\omega}}(q_0+h_1)$.

    Then we prove that $\pfis{(ij)^2}-\fdiag{ij}_\gamma\in \Ker(\Psi_{\bar{\omega}})$.
    By Lemma \ref{lem_morita_split_quat}, we have 
    \begin{align*}
        \Psi_{\bar{\omega}}(\fdiag{ij}_\gamma) &= (b_{\bar{\omega}})_*(\fdiag{ij}_{\gamma_F}) \\
        &\simeq \fdiag{-\Trd_{Q_F}(ij\cdot \bar{\omega})}\pfis{(ij)^2} \\
        &= \fdiag{-2(ij)^2}\pfis{(ij)^2} \\
        &= \fdiag{2}\pfis{(ij)^2}.
    \end{align*}

    Finally, let $q-h\in \Ker(\Psi_{\bar{\omega}})$. Then $q_F = \Psi_{\bar{\omega}}(h)$
    so by (\ref{eq_sequence_odd}) we have $\partial^1_\infty(q_F)=0$, so $q_{F_\infty}=0$.
    By \cite[Thm 34.7]{EKM}, $q = \pfis{(ij)^2}q'$ for some $q'\in W(k)$. Then 
    $\Psi_{\bar{\omega}}(h) = \Psi_{\bar{\omega}}(\fdiag{2}q'\fdiag{ij}_\gamma)$,
    so by injectivity of $\Psi_{\bar{\omega}}$ on $W^{-1}(Q,\gamma)$, 
    $h= \fdiag{2}q'\fdiag{ij}_\gamma$ and $q-h = 
    \fdiag{2}q'(\fdiag{2}\pfis{(ij)^2}-\fdiag{ij}_\gamma)$.
\end{proof}

\section{Modules of invariants}

In this section we define various modules of invariants, and prove some general statements 
that relate them, independently of any choice of generators.

\subsection{Definition of the modules}

We promote $\tld{W}^{-1}(Q,\gamma)$ to a functor $\mathbf{Field}_{/k}\to \mathbf{Ab}$ by 
setting 
\begin{equation}
    \underline{W^{-1}(Q,\gamma)}: K/k \mapsto W^{-1}(Q_K,\gamma_K)
\end{equation}
with the obvious scalar extension morphisms. 

Let us then define:
\begin{align}
    I_Q^{(r)} &= \Inv\left(H_Q^{(r)}, \underline{\tld{W}^{-1}(Q,\gamma)}\right) \\
    \bar{I}_Q^{(r)} &= \Inv\left(H_Q^{(r)}, \underline{\tld{W}^{-1}(Q,\gamma)}/(n_Q)\right) \\
    J_Q^{(r)} &= \Inv\left(\left(H_Q^{(1)}\right)^r, \underline{\tld{W}^{-1}(Q,\gamma)}\right) \\
    \bar{J}_Q^{(r)} &= \Inv\left(\left(H_Q^{(1)}\right)^r, \underline{\tld{W}^{-1}(Q,\gamma)}/(n_Q)\right).
\end{align}

These are all $\tld{W}^{-1}(Q,\gamma)$-modules. Since the functors where those invariants 
take values are $\Zd$-graded, this induces a $\Zd$-grading on those modules, and we will 
write 
\begin{equation}
    I_Q^{(r)} = \mbox{}^0 I_Q^{(r)} \oplus \mbox{}^1 I_Q^{(r)}
\end{equation}
for the corresponding decomposition into even and odd component, and likewise 
for the other modules. In the end, the module we are truly interested in is 
$\mbox{}^0 I_Q^{(r)} = \Inv(H_Q^{(r)}, W)$, but it is necessary to study the
full $I_Q^{(r)}$, which is reduced to the study of $\bar{I}_Q^{(r)}$, and in 
turn of $\bar{J}_Q^{(r)}$, which is determined by induction from $\bar{I}_Q^{(1)}$.
Note that by definition $I_Q^{(1)}= I_Q^{(1)}$ and $\bar{I}_Q^{(1)} = \bar{I}_Q^{(1)}$.

\begin{rem}\label{rem_odd_component}
    By Corollary \ref{cor_nq_trivial}, we have $\mbox{}^1 I_Q^{(r)} = \mbox{}^1 \bar{I}_Q^{(r)}$
    and $\mbox{}^1 J_Q^{(r)} = \mbox{}^1 \bar{J}_Q^{(r)}$.
\end{rem}

There is an obvious surjective natural transformation
\begin{equation}
    \anonfoncdef{\left(H_Q^{(1)}\right)^r}{H_Q^{(r)}}{(h_1,\dots,h_r)}{h_1\perp \dots \perp h_r}
\end{equation}
and an exact sequence 
\begin{equation}
    0 \to n_QW \to \underline{\tld{W}^{-1}(Q,\gamma)} \to 
    \underline{\tld{W}^{-1}(Q,\gamma)}/(n_Q) \to 0 
\end{equation}
which together induce a commutative diagram with exact lines and injective
vertical arrows:
\begin{equation}\label{eq_sequence_I_J}
    \begin{tikzcd}
        0 \rar & \Inv(H_Q^{(r)}, n_QW) \rar \dar & I_Q^{(r)} \rar \dar & \bar{I}_Q^{(r)} \dar \\
        0 \rar & \Inv\left(\left(H_Q^{(1)}\right)^r, n_QW\right) \rar & J_Q^{(r)} \rar   
        & \bar{J}_Q^{(r)} 
    \end{tikzcd}
\end{equation}

\subsection{Invariants in $n_QW$}

\begin{prop}\label{prop_constant_1}
    Every invariant in $\Inv(H_Q^{(1)}, n_QW)$ is constant.
\end{prop}

\begin{proof}
    Let $\alpha\in \Inv(H_Q^{(1)}, n_QW)$, which we can see as an invariant in 
    $\Inv(H_Q^{(1)}, W)$. Let $K = k(x,y,z)$. By \cite[Cor 27.13]{GMS}, since 
    $\tld{h} = \fdiag{\tld{\omega}}_{\gamma_K}$ is versal by Lemma \ref{lem_versal},
    $\alpha$ is constant if and only if $\alpha(\tld{h})\in W(K)$ is in the 
    image of $W(k)\to W(K)$. This is the case if and only if $\alpha(\tld{h})$
    is unramified along all hypersurfaces of $\mathbb{A}^3_k$ (see \cite[27.8]{GMS}).
    Since $\tld{h}$ corresponds to a torsor over $\mathbb{A}^3_k\setminus V(\tld{\omega}^2)$,
    $\alpha(\tld{h})$ can only be ramified along $V(\tld{\omega}^2)$ (\cite[Thm 27.11]{GMS}).
    But by hypothesis, $\alpha(\tld{h}) = n_{Q_K}q$ for some $q\in W(K)$. Let 
    $\mathcal{O} = k[x,y,z]_{(\tld{\omega}^2)}$ be the valuation ring of the 
    $\tld{\omega}^2$-adic valuation of $k[x,y,z]$.
    We can write $q = q_0 + \fdiag{\tld{\omega}^2}q_1$ with $q_0,q_1\in W(\mathcal{O})$,
    and since $\fdiag{\tld{\omega}^2}n_{Q_K}=n_{Q_K}$ because $-\tld{\omega}^2$ is 
    represented by $n_{Q_K}$ (it is the reduced norm of $\tld{\omega}$), we have
    $q\in W(\mathcal{O})$, ie it is unramified along $V(\tld{\omega}^2)$.
\end{proof}

We present a setting to use induction arguments for invariants.
Let $F: \mathbf{Field}_{/k}\to \mathbf{Set}$ and $A: \mathbf{Field}_{/k}\to \mathbf{Ab}$ 
be functors. We write $\underline{\Hom}$ for the internal $\Hom$ in 
a functor category, and $\underline{\Hom}_\Z$ for the internal $\Hom$
between two functors with values in abelian groups. By definition,
$\Inv_K(F,A) = \underline{\Hom}(F,A)(K)$.

Let $X$ be a finite set and $r\in \N$.
There is a canonical map 
\begin{equation}\label{eq_map_internal_hom}
    \underline{\Hom}_\Z(A^X,\underline{\Hom}(F,A)) \to 
    \underline{\Hom}_\Z(A^{X^r},\underline{\Hom}(F^r,A))
\end{equation}
defined through the isomorphisms $\underline{\Hom}_\Z(A^X,\underline{\Hom}(F,A))\simeq 
\underline{\Hom}(F\times X, \underline{\Hom}_\Z(A,A))$ and 
$\underline{\Hom}_\Z(A^{X^r},\underline{\Hom}(F^r,A))\simeq 
\underline{\Hom}(F^r\times X^r, \underline{\Hom}_\Z(A,A))$, where 
$X$ is seen as a constant functor, as well as the composition
\begin{equation}
    \underline{\Hom}(F\times X, \underline{\Hom}_\Z(A,A)) \to 
    \underline{\Hom}((F\times X)^r, \underline{\Hom}_\Z(A,A)^r) \to    
    \underline{\Hom}((F\times X)^r, \underline{\Hom}_\Z(A,A))
\end{equation}
where the first map is the diagonal embedding, and the second one is 
induced by the composition map $\underline{\Hom}_\Z(A,A)^r\to \underline{\Hom}_\Z(A,A)$
(that is, $(f_1,\dots,f_r)\mapsto f_1\circ\cdots\circ f_r$).

It is an easy fact to prove that under the canonical map (\ref{eq_map_internal_hom}),
an isomorphism $A^X\isom \underline{\Hom}(F,A)$ is sent to an isomorphism 
$A^{X^r},\underline{\Hom}(F^r,A)$. There are two special cases that are of 
interest to us, and we highlight them as lemmas.

\begin{lem}\label{lem_induction_constant}
    If the canonical map $A\to \underline{\Hom}(F,A)$ is an isomorphism,
    which means all invariants in $\Inv_K(F,A)$ are constant for 
    all $K/k$, then for any $r\in \N$, all invariants in $\Inv_K(F^r,A)$ 
    are also constant.
\end{lem}

\begin{proof}
    It is straightforward to see that (\ref{eq_map_internal_hom}) (with $X=\{\ast\}$) 
    sends the canonical map $A\to \underline{\Hom}(F,A)$ to the canonical 
    map $A\to \underline{\Hom}(F^r,A)$, as it corresponds to the constant map 
    $\underline{\Hom}(F\times X, \underline{\Hom}_\Z(A,A))$ to the identity of $A$,
    and the $r$-fold composition of the identity is the identity.
\end{proof}

\begin{lem}\label{lem_induction_free}
    Suppose $A$ is actually a functor to the category of commutative rings.
    For any finite family $(\alpha_x)_{x\in X}\in \Inv(F,A)^X$, we define 
    $(\alpha_{\bar{x}})_{\bar{x}\in X^r}\in \Inv(F^r,A)^{X^r}$, where 
    \[ \alpha_{(x_1,\dots,x_r)}(f_1,\dots,f_r) = \prod_{i=1}^r \alpha_{x_i}(f_i). \]
    Assume that $(\alpha_x)_{x\in X}$ is a strong basis of $\Inv(F,A)$, meaning that 
    it is a basis as an $A(k)$-module which remains an $A(K)$-basis of $\Inv_K(F,A)$ 
    for all $K/k$. Then $(\alpha_{\bar{x}})_{\bar{x}\in X^r}$ is a strong basis of 
    $\Inv(F^r,A)$.
\end{lem}

\begin{proof}
    The family $(\alpha_x)_{x\in X}$ defines a map $A^X\to \underline{\Hom}(F,A)$, 
    and one easily checks that (\ref{eq_map_internal_hom}) sends it to the map 
    $A^{X^r}\to \underline{\Hom}(F^r,A)$ defined by $(\alpha_{\bar{x}})_{\bar{x}\in X^r}$.
    Then we may conclude as $(\alpha_x)_{x\in X}$ is a strong basis if and only 
    if the corresponding $A^X\to \underline{\Hom}(F,A)$ is an isomorphism.
\end{proof}

Then we can use our induction properties to prove:

\begin{prop}\label{prop_constant}
    For any $r\in \N$, all invariants in $\Inv\left(\left(H_Q^{(1)}\right)^r, n_QW\right)$ and 
    $\Inv(H_Q^{(r)}, n_QW)$ are constant. An invariant in $I_Q^{(r)}$ is constant if 
    and only if its image in $\bar{I}_Q^{(r)}$ is constant.
\end{prop}

\begin{proof}
    From Proposition \ref{prop_constant_1}, the hypothesis of Lemma \ref{lem_induction_constant}
    is satisfied for $F = H_Q^{(1)}$ and $A=n_QW$, which settles the case of 
    $\Inv\left(\left(H_Q^{(1)}\right)^r, n_QW\right)$. Since $\Inv(H_Q^{(r)}, n_QW)$ is 
    embedded in $\Inv\left(\left(H_Q^{(1)}\right)^r, n_QW\right)$ (see \ref{eq_sequence_I_J}), 
    those invariants are also constant.

    The second statement is a straightfoward consequence of the first one, using the 
    top exact row in (see \ref{eq_sequence_I_J}).
\end{proof}

\subsection{Generic splitting}

Our description of $\bar{I}_Q^{(1)}$ will rely on generic 
splitting to reduce to our knowledge of invariants of 
$\Quad_2$.

Let $\alpha\in \bar{I}_Q^{(r)}$, and let $L/F$ be some 
extension. We define a function 
\begin{equation}
    \xi^{(r)}(\alpha): \Quad_{2r}(L)\to W(L)
\end{equation}
through the commutative diagram 
\begin{equation}
    \begin{tikzcd}
        H_{Q_F}^{(r)}(L) \rar{\alpha} \dar{(b_{\bar{\omega}_L})_*} & 
        \mathbf{\tld{W}^{-1}}(Q_L,\gamma_L) \dar{\Phi_{\bar{\omega}_L}} \\
        \Quad_{2r}(L) \rar{\xi^{(r)}(\alpha)} & W(L). 
    \end{tikzcd}
\end{equation}
This defines an invariant $\xi^{(r)}(\alpha)\in \Inv_F(\Quad_{2r}, W)$
because of the compatibility of Morita equivalences with scalar extensions,
expressed in the commutative diagram (\ref{eq_diagram_scalar_ext}).

We have thus defined a map
\begin{equation}
    \xi^{(r)}:  \bar{I}_Q^{(r)}\to \Inv_F(\Quad_{2r}, W).
\end{equation}

Similarly, we define 
\begin{equation}
    \zeta^{(r)}:  \bar{J}_Q^{(r)}\to \Inv_F(\Quad_{2}^r, W)
\end{equation}
through a diagram 
\begin{equation}
    \begin{tikzcd}
        H_{Q_F}^{(1)}(L)^r \rar{\alpha} \dar{(b_{\bar{\omega}_L})_*^r} & 
        \mathbf{\tld{W}^{-1}}(Q_L,\gamma_L) \dar{\Phi_{\bar{\omega}_L}} \\
        \Quad_{2}(L)^r \rar{\zeta^{(r)}(\alpha)} & W(L). 
    \end{tikzcd}
\end{equation}

By construction, the natural diagram 
\begin{equation}\label{eq_xi_zeta}
    \begin{tikzcd}
        \bar{I}_Q^{(r)} \rar{\xi^{(r)}} \dar & \Inv_F(\Quad_{2r}, W) \dar \\
        \bar{J}_Q^{(r)} \rar{\zeta^{(r)}} & \Inv_F(\Quad_{2}^r, W)
    \end{tikzcd}
\end{equation}
commutes.

Recall that $\Inv_F(\Quad_{2r}, W)$ is a $W(F)$-module, and 
through $\Psi_{\bar{\omega}}: \tld{W}^{-1}(Q,\gamma)\to W(F)$,
it is also a $\tld{W}^{-1}(Q,\gamma)$-module.

\begin{lem}\label{lem_xi_injective}
    The maps $\xi^{(r)}$ and $\zeta^{(r)}$ are morphisms of $\tld{W}^{-1}(Q,\gamma)$-modules.
    They are injective on $\mbox{}^0 \bar{I}_Q^{(r)}$, $\mbox{}^1 \bar{I}_Q^{(r)}$,
    $\mbox{}^0 \bar{J}_Q^{(r)}$ and $\mbox{}^1 \bar{J}_Q^{(r)}$.
\end{lem}

\begin{proof}
    The fact that $\xi$ is a $\tld{W}^{-1}(Q,\gamma)$-module morphism follows
    from the definition: let $\alpha\in \bar{I}_Q^{(r)}$ and $x\in \tld{W}^{-1}(Q,\gamma)$.
    Then for any extension $L/F$ and any $q\in \Quad_{2r}(L)$, if $h\in H_{Q}^{(r)}(L)$
    is such that $(b_{\bar{\omega}_L})_*(h)=q$, we have 
    \begin{align*}
        \xi^{(r)}(x\alpha)(q) &= \Phi_{\bar{\omega}_L}(x_L\alpha(h)) \\
        &= \Phi_{\bar{\omega}_L}(x_L)\Phi_{\bar{\omega}_L}(\alpha(h)) \\
        &= (\Phi_{\bar{\omega}_L}(x))_L\xi^{(r)}(\alpha)(q) \\
        &= (x\cdot \xi^{(r)}(\alpha))(q).
    \end{align*}
    If $\xi^{[r)}(\alpha)=0$, then let $K/k$ be an extension, and let 
    $h\in H_Q^{(r)}(K)$. Then by construction we have 
    \[ \Psi_{\bar{\omega}_{KF}}(\alpha(h)) = \xi^{(r)}(\alpha)((b_{\bar{\omega}_{KF}})_*(h_{KF}))=0. \]
    If $\alpha$ is in $\mbox{}^0 \bar{I}_Q^{(r)}$ or $\mbox{}^1 \bar{I}_Q^{(r)}$,
    then $\alpha(h)\in W(K)/(n_{Q_K})$ or $\alpha(h)\in W^{-1}(Q_K,\gamma_K)$; since
    $\Psi_{\omega_{KF}}$ is injective on these by Proposition \ref{prop_ker_psi},
    $\alpha(h)=0$, and therefore $\alpha=0$.

    The proofs are completely similar for $\zeta^{[r)}$ so we omit them.
\end{proof}

\section{Generators and relations}

In this section we give explicit presentations of our modules 
of invariants, the generators being given by the $\lambda^d$.

Precisely, for any $d,r\in \N$, the composition 
\begin{equation}
    H_{Q_K}^{(r)} \inj \tld{GW}^{-1}(Q_K,\gamma_K)\xrightarrow{\lambda^d} 
    \tld{GW}^{-1}(Q_K,\gamma_K) \to \tld{W}^{-1}(Q_K,\gamma_K)
\end{equation}
for all extensions $K/k$ form an invariant in $I_Q^{(r)}$, which 
we again denote $\lambda^d$. The compatibility with scalar extensions
is expressed by the fact that (\ref{eq_diagram_scalar_ext}) is a 
commutative diagram of pre-$\lambda$-rings.

The image of $\lambda^d$ in $\bar{I}_Q^{(r)}$ is 
written $\bar{\lambda}^d$. Note that if $d$ is even then 
$\lambda^d \in \mbox{}^0 \bar{I}_Q^{(r)}$, and if $d$ is odd 
then $\lambda^d \in \mbox{}^1 \bar{I}_Q^{(r)}$.

\subsection{Structure of $\bar{I}_Q^{(r)}$}

\begin{lem}\label{lem_xi_lambda}
    Let $r,d\in \N$. The morphism $\xi^{(r)}$ sends $\bar{\lambda}^d 
    \in \bar{I}_Q^{(r)}$ to $\lambda^d\in \Inv_F(Quad_{2r}, W)$.
\end{lem}

\begin{proof}
    The statement is equivalent to the commutativity of 
    \[ \begin{tikzcd}
        H_{Q_F}^{(r)}(L) \rar{\lambda^d} \dar{(b_{\bar{\omega}_L})_*} & 
        \mathbf{\tld{W}^{-1}}(Q_L,\gamma_L) \dar{\Phi_{\bar{\omega}_L}} \\
        \Quad_{2r}(L) \rar{\lambda^d} & W(L)
    \end{tikzcd} \]
    for all extensions $L/F$, and this is a simple consequence of the fact that 
    (\ref{eq_diagram_morita}) is a commutative diagram of pre-$\lambda$-rings.
\end{proof}

The crucial technical result is:

\begin{thm}\label{thm_inv_elem_nq}
    The $\tld{W}^{-1}(Q,\gamma)/(n_Q)$-module $\bar{I}_Q^{(1)}$
    is free, with basis $(\bar{\lambda}^0,\bar{\lambda}^1,\bar{\lambda}^2)$.
\end{thm}

\begin{proof}
    Technically, in order to use our results on generic splitting, we 
    need to distinguish the case where $k$ is quadratically closed, since 
    this is the only case where our description of $F$ using a closed 
    point of degree $2$ does not apply. When $k$ is quadratically closed,
    $Q$ is split, and the result is easily reduced using a Morita equivalence
    to the case of Witt invariants of $\Quad_2$, which is treated in \cite[Thm 27.16]{GMS}.
    We now exclude $k$ being quadratically closed for the rest of the proof.

    Let $\alpha\in \bar{I}_Q^{(1)}$. Then by \cite[Thm 27.16]{GMS}, $\xi^{(1)}(\alpha)$
    can be written $q_0+q_1\cdot \lambda^1+q_2\cdot \lambda^2$ for unique elements
    $q_i\in W(F)$. We study the residues of the $q_i$ with respect to 
    the valuations coming from closed points of $X_Q$.

    Let us write $K = k(x,y)$, and $KF = K\otimes_k F$.
    For any field extension $E/L$, we will write $\rho_{E/L}$
    for the scalar extension morphism $W(E)\to W(L)$.
    Let us consider 
    \begin{align*}
        h &= \fdiag{\omega\otimes 1}_{\gamma_{KF}}\in H_Q^{(1)}(KF)\\
        q &=  \Psi_{\bar{\omega}_{KF}}(h) \in \Quad_{2r}(KF) \\
        \theta &= \rho_{KF/F}(q_0)+q\rho_{KF/F}(q_1)+\fdiag{\det(q)}\rho_{KF/F}(q_2) \in W(KF).
    \end{align*}
    Then using Lemma \ref{lem_xi_lambda}, we have
    \[ \theta = \xi^{(1)}(\alpha)(q) = \Psi_{\bar{\omega}_{KF}}(\alpha(h_{KF})). \] 
    According to Lemma \ref{lem_morita_split_quat}, we can write $q = \fdiag{f}\pfis{g}$
    with 
    \begin{align*}
        g &= -\Delta \otimes 1 \in k[x,y]\otimes_k k \subset KF \\ 
          &= a x^2\otimes 1 + b y^2\otimes 1 - ab\otimes 1 \\
        f &= -\Trd_{Q_{KF}}(\omega\otimes \bar{\omega})\in k[x,y]\otimes k[\bar{x},\bar{y}]\subset KF \\ 
          &= -2(a x\otimes \bar{x} + b y\otimes \bar{y} - ab \otimes 1).
    \end{align*}
    
    Let $v$ be a valuation on $F$ which is either $v_p$ for 
    some $p\in X_Q^{(1)}$, with valuation ring $\mathcal{O}_v$,
    uniformizer $\pi$ and residue field $F_v$. We can have $p\in Y^{(1)}$ or $p=\infty$.
    Then $v$ extends naturally 
    to a discrete rank $1$ valuation $w$ on $KF$ with residue field $KF_v$,
    trivial on $K\otimes_k k$. 
    If we see $KF$ as the function field of $\mathbb{P}^2\times X_Q$, then $w$ 
    is the valuation associated to the hypersurface 
    $\mathbb{P}^2\times \{p\}$.
    
    We write $f = (1\otimes \pi^l) f'$ with $l\in \Z$ and $f'\in k[x,y]\otimes_k \mathcal{O}_v$.
    If $p \in Y^{(1)}$ we simply take $l=0$ and $f'=f$, and if $p=\infty$ 
    we can take $l=-1$. The image of $f'$ in $k[x,y]\otimes_k F_v$ is written 
    $\bar{f}$. Let us take $i\in \{1,2\}$. Then for any $m\in \{0,1,2\}$
    we have in $W(KF_v)$:
    \begin{align*}
        \partial^i_w(q_m) &= \rho_{KF_v/F_v}(\partial^i_v(q_0)) \\
        \partial^i_w(\fdiag{f}) &= \left\{ \begin{array}{lc}
            0 & \text{if $i=2$ and $p\in Y^{(1)}$, or $i=1$ and $p=\infty$} \\
            \fdiag{\bar{f}} & \text{if $i=1$ and $p\in Y^{(1)}$, or $i=2$ and $p=\infty$} \\
        \end{array}  \right. \\ 
        \partial^1_w(\fdiag{g}) &= \fdiag{-\Delta\otimes 1} \\
        \partial^1_w(\fdiag{g}) &= 0
    \end{align*}
    and therefore
    \begin{align}\label{eq_res_theta}
        \partial^i_w(\theta) &= \rho_{KF_v/F_v}(\partial^i_v(q_0)) 
        +\fdiag{\bar{f}}\pfis{-\Delta\otimes 1}\rho_{KF_v/F_v}(\partial^j_v(q_1))\\
        &+\fdiag{\Delta\otimes 1}\rho_{KF_v/F_v}(\partial^i_v(q_2)) \in W(KF_v)
    \end{align}
    where $j=i$ if $p\in Y^{(1)}$ and $j\neq i$ if $p = \infty$.
    
    We consider that discrete rank $1$ valuation $u'$ on $KF_v$ corresponding 
    to the hypersurface $X_Q\times \{p\}$ in $\mathbb{P}^2\times \{p\}$,
    with residue field $F\otimes_k F_v$. Then we take a valuation $u''$
    on $F\otimes_k F_v$ corresponding to any one of the two $F_v$-rational points 
    in $\{p\}\times \{p\}\simeq \Spec(F_v)\times \Spec(F_v)$. A crucial fact
    is that if $[f]\in F\otimes_k F_v$ is the image of $\bar{f}\in k[x,y]\otimes_k F_v$,
    then $u''([f])=1$. Indeed, the hypersurface of $Y\times Y$ defined by the 
    image of $f\in k[x,y]\otimes_k k[\bar{x},\bar{y}]$
    in $k[\bar{x},\bar{y}]\otimes_k k[\bar{x},\bar{y}] = k[Y]\otimes_k k[Y]$ is 
    the diagonal embedding $Y\to Y\times Y$. Thus the valuation of $[f]$ 
    at any point in the intersection of the diagonal of $X_Q\times X_Q$ and $X_Q\times \{p\}$,
    which is precisely $\{p\}\times \{p\}$, is $1$.
    This means that if $u = u''\circ u'$ is the composed valuation on 
    $KF_v$, with value group $\Z^2$ and residue field $F_v$, then 
    $u(-\Delta\otimes 1) = (1,0)$ and $u(\bar{f}) = (0,1)$. 
    One also sees that $k\otimes_k F_v\subset \mathcal{O}_u^\times$.
    
    Using 
    $\Delta\otimes 1$ and $\bar{f}$ as uniformizers, we get a residue map 
    \[ \partial_u: W(KF_v) \to W(F_v)[(\Zd)^2] \]
    such that, using (\ref{eq_res_theta}):
    \begin{equation}\label{eq_res_theta_u}
        \partial_u(\partial^i_w(\theta)) = (\partial^i_v(q_0), \partial^j_v(q_1), \partial^i_v(q_2),
        \partial^j_v(q_1)) \in W(F_v)[(\Zd)^2].
    \end{equation}
    
    Since $\theta$ is in the image of $\Psi_{\bar{\omega}_{KF}}$,
    by the exact sequences (\ref{eq_sequence_even}) and (\ref{eq_sequence_odd})
    we have that $\partial^i_w(\theta)=0$ when $i=2$ and $p\in Y^{(1)}$,
    which yields
    \[ \partial^2_v(q_0) = \partial^2_v(q_1) = \partial^2_v(q_2) = 0, \]
    in other words $q_0,q_1,q_2\in W_0(F)$. From this point we assume that 
    $p=\infty$.

    Now assume $\alpha\in \mbox{}^0 \bar{I}_Q^{(r)}$. Then $\alpha(h)\in W(K)/(n_{Q_K})$
    so by (\ref{eq_sequence_even}) we have $\partial^2_w(\theta)=0$, which yields 
    with (\ref{eq_res_theta_u}):
    \[ \partial^2_v(q_0) = \partial^1_v(q_1) = \partial^2_v(q_2) = 0, \]
    so $q_0 = \Psi_{\bar{\omega}}(\phi_0)$, $q_1 = \Psi_{\bar{\omega}}(h_1)$
    and $q_2 = \Psi_{\bar{\omega}}(\phi_2)$, with $\phi_0,\phi_2\in W(k)$
    uniquely determined modulo $n_Q$, and $h_1\in W^{-1}(Q,\gamma)$ uniquely 
    determined. This shows that 
    \[ \xi^{(1)}(\alpha)=\xi^{(1)}(\phi_0 + h_1\bar{\lambda}^1 + \phi_2\bar{\lambda}^2) \]   
    and by injectivity of $\xi^{(1)}$ on $\mbox{}^0 \bar{I}_Q^{(r)}$ (Lemma \ref{lem_xi_injective}),
    we deduce that 
    \[ \alpha = \phi_0 + h_1\bar{\lambda}^1 + \phi_2\bar{\lambda}^2, \]
    for unique $\phi_0,\phi_2\in W(k)/(n_Q)$ and a unique $h_1\in W^{-1}(Q,\gamma)$.  

    Likewise, assume that $\alpha\in \mbox{}^1 \bar{I}_Q^{(r)}$. This time 
    $\alpha(h)\in W^{-1}(Q_K,\gamma_K)$ so by (\ref{eq_sequence_even}) we have 
    $\partial^1_w(\theta)=0$, then 
    \[ \partial^1_v(q_0) = \partial^2_v(q_1) = \partial^1_v(q_2) = 0, \]
    and again by injectivity of $\xi^{(1)}$ on $\mbox{}^1 \bar{I}_Q^{(r)}$,
    we have unique $h_0,h_2\in W^{-1}(Q,\gamma)$ and a unique $\phi_1\in W(k)/(n_Q)$
    such that 
    \[ \alpha = h_0 + \phi_1\bar{\lambda}^1 + h_2\bar{\lambda}^2. \]

    In the general case, we write $\alpha = \alpha_0 + \alpha_1$ with 
    $\alpha_0\in \mbox{}^0 \bar{I}_Q^{(r)}$ and 
    $\alpha_1\in \mbox{}^1 \bar{I}_Q^{(r)}$, and the two previous points 
    show that there are unique $x_0, x_1, x_2\in \tld{W}^{-1}(Q,\gamma)/(n_Q)$
    such that 
    \[ \alpha = x_0 + x_1\bar{\lambda}^1 + x_2\bar{\lambda}^2. \]
\end{proof}

We get the more general case by induction. First:

\begin{coro}\label{cor_bar_j}
    For any $r\in \N$,
    the $\tld{W}^{-1}(Q,\gamma)/(n_Q)$-module $\bar{J}_Q^{(r)}$
    is free, with basis $(\bar{\lambda}^{\underline{d}})_{\underline{d}\in \{0,1,2\}^r}$,
    where 
    \[ \bar{\lambda}^{\underline{d}}(h_1,\dots,h_r) = \prod_{i=1}^r \bar{\lambda}^{d_i}(h_i).\]

    If $\alpha = \sum_{\underline{d}} x_{\underline{d}}\bar{\lambda}^{\underline{d}}$,
    then $\alpha\in \mbox{}^0 \bar{J}_Q^{(r)}$ if and only if $x_{\underline{d}}\in W(k)/(n_Q)$
    for all $\underline{d}$ with $|\underline{d}| = \sum_i d_i$ is even, and 
    $x_{\underline{d}}\in W^{-1}(Q,\gamma)$ when $|\underline{d}|$ is odd. Similarly,
    $\alpha\in \mbox{}^1 \bar{J}_Q^{(r)}$ if and only if $x_{\underline{d}}\in W^{-1}(Q,\gamma)$ 
    when $|\underline{d}|$ is even and $x_{\underline{d}}\in W(k)/(n_Q)$ 
    when $|\underline{d}|$ is odd.
\end{coro}

\begin{proof}
    This is a direct application of Lemma \ref{lem_induction_free}, since 
    $\bar{I}_Q^{(1)} = \bar{J}_Q^{(1)}$.

    Clearly, $\bar{\lambda}^{\underline{d}}\in \mbox{}^0 \bar{J}_Q^{(r)}$ 
    when $|\underline{d}|$ is even, and $\bar{\lambda}^{\underline{d}}\in \mbox{}^1 \bar{J}_Q^{(r)}$ 
    when $|\underline{d}|$ is odd. So if $\alpha = \sum_{\underline{d}} x_{\underline{d}}\bar{\lambda}^{\underline{d}}$,
    and $x_{\underline{d}} = q_{\underline{d}}+h_{\underline{d}}$, then $\alpha = \alpha_0 + \alpha_1$ with 
    \[  \alpha_0 = \sum_{|\underline{d}| \text{even}} q_{\underline{d}}\bar{\lambda}^{\underline{d}}
    + \sum_{|\underline{d}| \text{odd}} h_{\underline{d}}\bar{\lambda}^{\underline{d}}  \]
    and 
    \[  \alpha_1 = \sum_{|\underline{d}| \text{even}} h_{\underline{d}}\bar{\lambda}^{\underline{d}}
    + \sum_{|\underline{d}| \text{odd}} q_{\underline{d}}\bar{\lambda}^{\underline{d}}  \]
    and since $\alpha_0 \in \mbox{}^0 \bar{J}_Q^{(r)}$ and $\alpha_1 \in \mbox{}^1 \bar{J}_Q^{(r)}$,
    this is the unique decomposition, and the result follows.
\end{proof}

We can then deduce:

\begin{thm}\label{thm_inv_nq}
    For any $r\in \N$,
    the $\tld{W}^{-1}(Q,\gamma)/(n_Q)$-module $\bar{I}_Q^{(r)}$
    is free, with basis $(\bar{\lambda}^0,\dots,\bar{\lambda}^{2r})$.
\end{thm}

\begin{proof}
    Let $\alpha\in \bar{I}_Q^{(r)}$, and $\beta$ its image in 
    $\bar{J}_Q^{(r)}$. We write 
    \[  \beta = \sum_{\underline{d}}x_{\underline{d}}\bar{\lambda}^{\underline{d}}.  \]
    We can also write 
    \[ \xi^{(r)}(\alpha) = \sum_{d=0}^{2r} y_d \lambda^d \]   
    for unique elements $y_d\in W(F)$.
    The image of $\lambda^d\in \Inv_F(\Quad_{2r}, W)$ in 
    $\Inv_F(\Quad_{2}^r, W)$ is $\sum_{|\underline{d}|=d} \lambda^{\underline{d}}$,
    which implies by the diagram (\ref{eq_xi_zeta}):
    \[ \zeta^{(r)}(\beta) = \sum_{\underline{d}}\psi_{\bar{\omega}}(x_{|\underline{d}|})\bar{\lambda}^{\underline{d}}  \]
    and therefore by uniqueness of the decomposition:
    \[  x_{\underline{d}} = \Psi_{\bar{\omega}}(x_{|\underline{d}|})  \]
    for all $\underline{d}$.

    If $\alpha\in \mbox{}^0 \bar{I}_Q^{(r)}$ or $\alpha\in \mbox{}^1 \bar{I}_Q^{(r)}$,
    then each $x_{\underline{d}}$ is in $W(k)/(n_Q)$ or $W^{-1}(Q,\gamma)$ by 
    Corollary \ref{cor_bar_j}, so by injectivity of $\Psi_{\bar{\omega}}$ on these 
    groups (Proposition \ref{prop_ker_psi}), $x_{\underline{d}}$ only depends on 
    $|\underline{d}|$. If $x_{|\underline{d}|}$ is this common value, 
    then 
    \[  \beta = \sum_{d=0}^{2r} x_d\left(\sum_{|\underline{d}|=d} \bar{\lambda}^{\underline{d}} \right).  \]
    Since $\sum_{|\underline{d}|=d} \bar{\lambda}^{\underline{d}}$ is the image of $\bar{\lambda}^d\in \bar{I}_Q^{(r)}$
    in $\bar{J}_Q^{(r)}$, by injectivity of $\bar{I}_Q^{(r)}\to \bar{J}_Q^{(r)}$ we get 
    \[  \alpha = \sum_{d=0}^{2r} x_d \bar{\lambda}^d  \]
    for unique elements $x_d$ (with $x_d\in W(k)/(n_Q)$ if $d$ is even and $x_d\in W^{-1}(Q,\gamma)$
    if $d$ is odd).

    Similarly, if $\alpha\in \mbox{}^1 \bar{I}_Q^{(r)}$ we deduce that $\alpha = \sum_{d=0}^{2r} x_d \bar{\lambda}^d$
    for unique elements $x_d$, with $x_d\in W(k)/(n_Q)$ if $d$ is odd and $x_d\in W^{-1}(Q,\gamma)$
    if $d$ is even.

    In general, we decompose $\alpha$ as $\alpha_0+\alpha_1$ with $\alpha_i\in \mbox{}^i \bar{I}_Q^{(r)}$
    to get the expected result.
\end{proof}

\begin{coro}\label{cor_sequence_I_J}
    The commutative diagram (\ref{eq_sequence_I_J}) with exact lines can be refined as 
    \[ \begin{tikzcd}
            0 \rar & n_QW \rar \dar & I_Q^{(r)} \rar \dar & \bar{I}_Q^{(r)} \dar \rar & 0 \\
            0 \rar &  n_QW \rar & J_Q^{(r)} \rar  & \bar{J}_Q^{(r)} \rar & 0
        \end{tikzcd} \]
\end{coro}

\begin{proof}
    The $n_QW$ in the diagram come from Proposition \ref{prop_constant}. The surjectivities
    follows from Theorem \ref{thm_inv_nq} and Corollary \ref{cor_bar_j}, since the generators 
    $\bar{\lambda^d}\in \bar{I}_Q^{(r)}$ and $\bar{\lambda^{\underline{d}}}\in \bar{J}_Q^{(r)}$ are 
    the images of $\lambda^d\in I_Q^{(r)}$ and $\lambda^{\underline{d}}\in J_Q^{(r)}$.
\end{proof}

\subsection{Structure of $I_Q^{(r)}$}

The structure of $I_Q^{(r)}$ is slightly more complicated than that of $I_Q^{(r)}$, 
as it is not a free module. The (only) obstruction to being free comes from the 
following simple fact:

\begin{prop}\label{prop_nq_lambda}
    In $I_Q^{(r)}$, for any $d\in \N$ we have that 
    \[ n_Q\cdot \lambda^d = \left\{ \begin{array}{lc}
        0 & \text{if $d$ is odd} \\
        \binom{r}{d/2} & \text{if $d$ is even.}
    \end{array}  \right. \]
\end{prop}

\begin{proof}
    The statement when $d$ is odd follows from Corollary \ref{cor_nq_trivial}.

    When $d$ is even, we first treat the case $r=1$, where we only need to 
    look at $d=2$. Let $K/k$ be some extension, and $h=\fdiag{z}_{\gamma_K}\in H_Q^{(1)}(K)$.
    Then $\lambda^2(h)=\fdiag{\Nrd_{Q_K}(z)}$, and by definition $\Nrd_{Q_K}(z)$
    is represented by the $2$-fold Pfister form $n_{Q_K}$, so 
    $n_{Q_K}\lambda^2(h)=n_{Q_K}$.

    Now if $r\N^*$, and $h=\fdiag{z_1,\dots, z_r}_\gamma$, we have 
    \[ n_Q\lambda^{2d}(h) = \sum_{d_1+\cdots+d_r=d}n_Q \lambda^{d_1}(\fdiag{z_1})\cdots\lambda^{d_1}(\fdiag{z_1}). \]
    In the sum, all $d_i$ are $0$, $1$ or $2$; if one of them is $1$ then
    the term is $0$ from the case where $d$ is odd. When all of them are 
    $0$ or $2$, the case $r=1$ shows that the term is equal to $n_Q$.
    A simple counting argument shows that there are $\binom{r}{d}$ non-zero 
    terms.
\end{proof}

This leads us to introduce
\begin{equation}
    \foncdef{\chi^{(r)}}{\tld{W}^{-1}(Q,\gamma)^{2r+1}}{\tld{W}^{-1}(Q,\gamma)}
    {(x_0,\dots,x_{2r})}{\sum_{i=0}^r \binom{r}{i}x_{2i}.}
\end{equation}

We can then state the main result of this article.

\begin{thm}\label{thm_inv}
    The $\tld{W}^{-1}(Q,\gamma)$-module $I_Q^{(r)}$ is generated
    by $(\lambda^0,\dots,\lambda^{2r})$.

    If $(x_0,\dots,x_{2r})\in \tld{W}^{-1}(Q,\gamma)^{2r+1}$, then 
    the invariant $\alpha = \sum_{d=0}^{2r}x_d\lambda^i$ is constant 
    if and only if $x_d\in n_QW(k)$ for all $d>0$, and in that case 
    the constant is $\chi^{(r)}(x_0,\dots,x_{2r})$.
\end{thm}

\begin{proof}
    From the exact sequence of $\tld{W}^{-1}(Q,\gamma)$-modules
    \[ 0 \to n_QW(k) \to I_Q^{(r)} \to \bar{I}_Q^{(r)} \to 0 \]
    of Corollary \ref{cor_sequence_I_J}, and the fact that 
    $\bar{I}_Q^{(r)}$ is generated by the $\bar{\lambda}^d$,
    we deduce that $I_Q^{(r)}$ is generated by the $\lambda^d$.

    Let $\alpha = \sum_{d=0}^{2r}x_d\lambda^d \in I_Q^{(r)}$,
    and let $\bar{\alpha}$ be its image in $\bar{I}_Q^{(r)}$.
    Bu Proposition \ref{prop_constant}, $\alpha$ is constant 
    if and only if $\bar{\alpha}$ is, and since $\bar{I}_Q^{(r)}$
    is free and the constant invariants in $\bar{I}_Q^{(r)}$ 
    are the submodule generated by $\bar{\lambda}^0$, $\bar{\alpha}$
    is constant if and only if the class of $x_d$ in $\tld{W}^{-1}(Q,\gamma)$
    is $0$ for $d>0$. By Corollary \ref{cor_nq_trivial}, this 
    means that $x_d\in n_Q Q(k)$.

    If we write $x_d = n_Qy_d$ for each $d>0$, then by Proposition
    \ref{prop_nq_lambda} we have 
    \begin{align*}
        \alpha &= x_0 + \sum_{d=1}^{2r}y_d\cdot n_Q\lambda^d \\
        &= x_0 + \sum_{i=1}^{r}y_{2i}\cdot n_Q\binom{r}{i} \\ 
        &= \chi^{(r)}(x_0,\dots,x_{2r}).
    \end{align*}
\end{proof}

This theorem gives a presentation of $I_Q^{(r)}$ as a $\tld{W}^{-1}(Q,\gamma)$-module:
\begin{equation}
    I_Q^{(r)} = \left\langle \lambda^0,\dots,\lambda^{2r} \,|\, \forall 0\ppq i\ppq r, n_Q\lambda^{2i} = n_Q\binom{r}{i}\lambda^0 \right\rangle.
\end{equation}

Of course when $Q$ is split $n_Q=0$ and the module is free.

\bibliographystyle{plain}
\bibliography{witt_invariants_quaternionic_forms}

\end{document}

%% file: canevas_anglais.tex
\usepackage[utf8]{inputenc}
\usepackage[english]{babel}
\usepackage[T1]{fontenc}
\usepackage{amsmath}
\usepackage{amsfonts}
\usepackage{amssymb}
\usepackage{amsthm}
\usepackage{stmaryrd}
\usepackage{tikz}
\usepackage{tikz-cd}
\usepackage{enumitem}
\usepackage{todonotes} 
\usepackage{graphicx}
\usepackage{hyperref}

\DeclareMathOperator{\Spec}{Spec}

\DeclareMathOperator{\Hom}{Hom}
\DeclareMathOperator{\End}{End}

\DeclareMathOperator{\Ker}{Ker}

\DeclareMathOperator{\Trd}{Trd}
\DeclareMathOperator{\Nrd}{Nrd}
\DeclareMathOperator{\Br}{Br}

\DeclareMathOperator{\Iso}{Iso}

\DeclareMathOperator{\Inv}{Inv}

\DeclareMathOperator{\Id}{Id}

\DeclareMathOperator{\Quad}{Quad}

\newcommand{\inj}{\hookrightarrow}

\newcommand{\isom}{\stackrel{\sim}{\rightarrow}}

\newcommand{\Z}{\mathbb{Z}}
\newcommand{\N}{\mathbb{N}}

\newcommand{\pfis}[1]{\langle\!\langle #1\rangle\!\rangle}
\newcommand{\To}{\longrightarrow}

\newcommand{\fdiag}[1]{\langle #1\rangle}

\newcommand{\tld}{\widetilde}
\newcommand{\eps}{\varepsilon}

\newcommand{\ppq}{\leqslant}

\newcommand{\Zd}{\Z/2\Z}

\renewcommand{\phi}{\varphi}
\renewcommand{\bar}{\overline}

\newcommand{\foncdef}[5]{\begin{array}{rrcl}
#1 : & #2 & \To & #3 \\
 & #4 & \longmapsto & #5
\end{array}}

\newcommand{\anonfoncdef}[4]{\begin{array}{rcl}
#1 & \To & #2 \\
#3 & \longmapsto & #4
\end{array}}

\newtheorem{thm}{Theorem}[section]
\newtheorem{prop}[thm]{Proposition}
\newtheorem{coro}[thm]{Corollary}
\newtheorem{lem}[thm]{Lemma}

\theoremstyle{definition}
\newtheorem{rem}[thm]{Remark}